\documentclass{amsart}

\usepackage{amssymb}
\usepackage{mathtools}
\usepackage[margin=1in]{geometry}
\usepackage{comment}

\theoremstyle{plain}
\newtheorem{theorem}{Theorem}[section]
\newtheorem{proposition}[theorem]{Proposition}
\newtheorem{lem}[theorem]{Lemma}
\newtheorem{cor}[theorem]{Corollary}
\newtheorem{rem}[theorem]{Remark}

\theoremstyle{definition}
\newtheorem*{acknowledgement}{Acknowledgements}

\newcommand{\PP}{\mathbb{P}}

\newcommand{\EE}{\mathbb{E}}
\newcommand{\NN}{\mathbb{N}}
\newcommand{\RR}{\mathbb{R}}
\newcommand{\CC}{\mathbb{C}}
\newcommand{\Lower}{D}

\numberwithin{equation}{section}
\mathtoolsset{showonlyrefs}

\begin{document}
	
	\title[On an Erd\H{o}s--Kac-type conjecture]{On an Erd\H{o}s--Kac-type conjecture of Elliott}
	\author{Ofir Gorodetsky}
	\address{Mathematical Institute, University of Oxford, Oxford, OX2 6GG, UK}
	\email{gorodetsky@maths.ox.ac.uk}
	\thanks{This project has received funding from the European Research Council (ERC) under the European Union's Horizon 2020 research and innovation programme (grant agreement No 851318).}
	\author{Lasse Grimmelt}
	\address{Mathematical Institute, University of Oxford, Oxford, OX2 6GG, UK}
	\email{lasse.grimmelt@maths.ox.ac.uk}

\begin{abstract}
Elliott and Halberstam proved that $\sum_{p<n} 2^{\omega(n-p)}$ is asymptotic to $\phi(n)$. In analogy to the Erd\H{o}s--Kac Theorem, Elliott conjectured that if one restricts the summation to primes $p$ such that $\omega(n-p)\le 2 \log \log n+\lambda(2\log \log n)^{1/2}$ then the sum will be asymptotic to $\phi(n)\int_{-\infty}^{\lambda} e^{-t^2/2}dt/\sqrt{2\pi}$. We show that this conjecture follows from the Bombieri--Vinogradov Theorem. We further prove a related result involving Poisson--Dirichlet distribution, employing deeper lying level of distribution results of the primes.
\end{abstract}
\maketitle
\section{Introduction and motivation}
Let $\omega(m) = \sum_{p \mid m}1$ be the prime divisor function. Elliott and Halberstam proved that \cite{ElliottH}
\begin{equation}\label{eq:elliott}
\sum_{p<n} 2^{\omega(n-p)} = \phi(n) + O\left(\frac{n}{\log n}(\log \log n)^2\right)
\end{equation}
holds for $n \ge 3$. Here and later $p$ denotes a prime and $\phi$ is Euler's totient function. Elliott gave a talk at the meeting held in Urbana--Champaign, June 5–7, 2014, in memory of Paul and Felice Bateman, and Heini Halberstam, where he revisited his works with Halberstam. A telegraphic representation of the talk was published, which contains the following conjecture \cite[Conj.~B]{Elliott}: for each real $\lambda$,	
\begin{equation}\label{eq:conjb} \sum_{\substack{p <n\\ \omega(n-p)\le 2\log \log n+ \lambda (2\log \log n)^{1/2}}}2^{\omega(n-p)} \to \frac{\phi(n)}{\sqrt{2\pi}}\int_{-\infty}^{\lambda}e^{-u^2/2}du
\end{equation}
as $n \to \infty$. In this note we establish the conjecture.
\begin{theorem}\label{thm:e}
The asymptotic relation \eqref{eq:conjb} holds for each real $\lambda$. 
\end{theorem}
Theorem~\ref{thm:e} can be restated as follows. Let $X_n$ be a random variable taking the value $m \in \{1,2,\ldots,n\}$ with probability proportional to $2^{\omega(m)} \mathbf{1}_{n-m \textup{ a prime}}$. Then
\begin{equation}\label{eq:prob}
	\frac{\omega(X_n)-2\log \log n}{(2\log \log n)^{1/2}} \overset{d}{\longrightarrow} N(0,1), \qquad n \to \infty,
\end{equation}
where $N(0,1)$ is the standard gaussian distribution and the arrow indicates convergence in distribution. The proof of Theorem~\ref{thm:e} will follow from 
\begin{proposition}\label{prop:titch}
Fix $k,B \ge 1$. For any $1 \le R \le n^{1/100}$ and $S \subseteq \NN$ we have
	\begin{equation}\label{eq:primes} 
		\sum_{\substack{a \le R,\, a \in S\\ \omega(a)\le k\\ (a,n)=1}}\mu^2(a)\bigg| \PP( a \textup{ divides } X_n) - \frac{2^{\omega(a)}}{a}\bigg| \ll_{k,B} \frac{(\log \log n)^3 +\log R}{\log n}	\sum_{\substack{a \le R,\, a \in S\\ \omega(a)\le k\\(a,n)=1}}\frac{\mu^2(a)}{a} +\frac{1}{\log^B n}.
	\end{equation}
\end{proposition}
The reduction of Theorem~\ref{thm:e} to Proposition~\ref{prop:titch} is explained in \S\ref{sec:reduc}. The proof of Proposition~\ref{prop:titch} is given in \S\ref{sec:titch} and uses the classical tools utilized in \cite{ElliottH} (the Brun--Titchmarsh and Bombieri--Vinogradov Theorems). Proposition~\ref{prop:titch} has little content when $\log R \asymp \log n$. This is remedied by the following proposition, proved using recent work on primes in arithmetic progressions \cite{Blomer}.
\begin{proposition}\label{prop:titchstronger}
There is an absolute constant $\eta>0$ with the following property. For every $k \ge 1$,	\begin{equation}\label{eq:primesstronger}
		\sum_{\substack{a \le n^{\eta}\\ \omega(a)\le k \\ (a,n)=1} }\mu^2(a)\bigg| \PP( a \textup{ divides } X_n) - \frac{2^{\omega(a)}}{a}+\frac{2\Lower_a}{\phi(a)\log n}\bigg| \ll_k\frac{(\log \log n)^{C_k}}{\log n}
	\end{equation}
where $C_k$ is a positive constant depending only on $k$,
\begin{equation}\label{eq:gdef}
	\Lower_a:= \sum_{d\mid a^{\infty}} g(d)\log d, \qquad g(d):=\prod_{p^k\mid \mid d} (-1)^{k-1}(p-2)(p-1)^{-k}.
\end{equation}
\end{proposition}
Since $\sum_{d \mid a^{\infty}} |g(d)| =2^{\omega(a)-\mathbf{1}_{2 \mid a}}$, we have
\begin{equation}\label{eq:dabnd}
 |\Lower_a| \le \sum_{p\mid a} \sum_{i\ge 1}\log (p^i) \sum_{p^i \mid \mid d \mid a^{\infty}}|g(d)| \le 2^{\omega(a)}\sum_{p\mid a} \sum_{i\ge 1}\log (p^i) |g(p^i)|\ll 2^{\omega(a)}\log a.
\end{equation}
Proposition~\ref{prop:titchstronger} is proved in \S\ref{sec:titchstronger}.
The lower order term in \eqref{eq:primesstronger} leads to the following application. Given an integer $m>1$ we write it as
$m = \prod_{i=1}^{\Omega(m)}p_i(m)$ where $\Omega(m)$ is the number of prime factors of $m$ (multiplicity counted) and $p_1(m)\ge \ldots \ge p_{\Omega(m)}$. We define 
\begin{equation}\label{eq:defV} V(m):= \left(\frac{\log p_i(m)}{\log m}\right)_{i \ge 1}
\end{equation}
where $p_i(m) := 1$ for $i>\Omega(m)$. This vector sums to $1$ and its entries are nonincreasing. Let $\textup{PD}(\theta)$ be the Poisson--Dirichlet distribution with parameter $\theta$ (see Chapter 1 of \cite{BillingsleyBook} for its definition). 
\begin{cor}\label{cor:PD}
If $V(X_n)$ tends in distribution to $\textup{PD}(\theta)$ for some $\theta>0$ then $\theta=2$.
\end{cor}
Let $\widetilde{X}_n$ be the random variable taking the value $m \in \{1,2,\ldots, n\}$ with probability proportional to $2^{\omega(m)}$. As observed by Elliott \cite{Elliott},
\begin{equation}\label{eq:uprime}	\frac{\omega(\widetilde{X}_n)-2\log \log n}{(2 \log \log n)^{1/2}} \overset{d}{\longrightarrow} N(0,1), \qquad n \to \infty.
\end{equation}
See \S\ref{sec:motivation} for a self-contained proof of \eqref{eq:uprime}. As shown in \cite{EG}, $V(\widetilde{X}_n)$ tends to $\textup{PD}(2)$. We expect Theorem~\ref{thm:e} and Corollary~\ref{cor:PD} to hold because $X_n$ is distributed like $\widetilde{X}_n$, only with support restricted to shifted primes which are believed to behave in most ways like random integers.

The structure of the paper is as follows. In the next short section we show how Corollary~\ref{cor:PD} follows from Proposition~\ref{prop:titchstronger}. Afterwards in \S\ref{sec:reduc} we reduce the conjecture of Elliott we prove, Theorem~\ref{thm:e} to Proposition~\ref{prop:titch}. In \S\ref{sec:titch} we show how Proposition~\ref{prop:titch} follows from the classical Bombieri--Vinogradov and Brun--Titchmarsh Theorems. At this point the proof of Theorem~\ref{thm:e} is complete. In the penultimate \S\ref{sec:titchstronger} we go beyond the Bombieri--Vinogradov range by employing a result in \cite{Blomer}, based on spectral theory of automorphic forms, to prove Proposition~\ref{prop:titchstronger}. As mentioned in the last paragraph, we conclude the paper with a short proof of \eqref{eq:uprime}.

\section{Proof of Corollary~\ref{cor:PD}}
We now show how Corollary~\ref{cor:PD} follows from Proposition~\ref{prop:titchstronger}.  Suppose $V(X_n)\overset{d}{\longrightarrow}\textup{PD}(\theta)$, and in particular $V(X_n)_i\overset{d}{\longrightarrow} \textup{PD}(\theta)_i$ for all $i$. Since $2^{\omega(n-p)}=n^{o(1)}$, we have $\mathbb{P}(X_n \le T) \ll T/n^{1+o(1)}$, so that $\log X_n/\log n$ tends to $1$ in probability. This implies $\widetilde{V}(X_n)_i\overset{d}{\longrightarrow} \textup{PD}(\theta)_i$  for $\widetilde{V}(X_n)_i:=  \log p_i(X_n)/\log n$. By \cite[Thm.~2.1]{Handa}, 
\[ \lim_{n \to \infty}\EE \sum_{i=1}^{\infty} \mathbf{1}_{\widetilde{V}(X_n)_i\in [a,b]} =\theta\int_{a}^{b} \frac{(1-t)^{\theta-1}}{t} dt\]
for all $0<a<b<1$ (the sum appears infinite but can be truncated at $i \le\lfloor 1/a\rfloor$ since $\sum_{i\ge 1} \widetilde{V}(X_n)_i \le 1$). This is simplified as
\begin{equation}\label{eq:corr}
\lim_{n \to \infty}\sum_{p \in [n^a,n^b]}\sum_{k \ge 1} \mathbb{P}(p^k \text{ divides } X_n)  =\theta\int_{a}^{b} \frac{(1-t)^{\theta-1}}{t}dt.
\end{equation}
Since $\mathbb{P}(d \text{ divides }X_n)\ll n^{o(1)}/d$, \eqref{eq:corr} is equivalent to \begin{equation}\label{eq:corr2}
\lim_{n \to \infty}\sum_{n^a\le p \le n^b} \mathbb{P}(p  \text{ divides }  X_n)  =\theta\int_{a}^{b} \frac{(1-t)^{\theta-1}}{t}dt.
\end{equation}
By Proposition~\ref{prop:titchstronger} with $k=1$ we have
\begin{equation}
\lim_{n \to \infty}\sum_{n^a\le p \le n^b} \mathbb{P}(p  \text{ divides }  X_n)  = 2\int_{a}^{b} \frac{1-t}{t}dt
\end{equation}
by Mertens' theorem \cite[Thm.~2.7(d)]{MV}, as long as $b<\eta$. This is consistent with \eqref{eq:corr2} only if $\theta=2$.

\begin{rem}
More generally, let $k \ge 1$ and a continuous function $f \colon \RR^k \to \CC$ such that the (closure of the) support of $f$ is $\subseteq \{ (x_1,\ldots,x_k): x_1,\ldots,x_k >0, \, x_1+\ldots+x_k < \eta\}$. Then with additional work one can compute 
\[ \lim_{n \to \infty} \EE \sum_{j_1,\ldots,j_k \text{ distinct}} f( V(X_n)_{j_1},\ldots,V(X_n)_{j_k})\]
using Proposition~\ref{prop:titchstronger} and show that it coincides with 
\[ \EE \sum_{j_1,\ldots,j_k \text{ distinct}} f( L_{j_1},\ldots,L_{j_k})\]
where $(L_1,L_2,\ldots)$ is distributed according to $\textup{PD}(2)$. In the language of Bharadwaj and Rodgers \cite{bharadwaj2024large}, this means we can compute the correlation functions of large prime factors of $X_n$ and show they tend to those of the Poisson--Dirichlet
process with parameter $\theta=2$, at least for test functions with restricted support.
\end{rem}

\section{Reduction}\label{sec:reduc}
We now reduce Theorem~\ref{thm:e} to Proposition~\ref{prop:titch}. In other words, we make precise the well-known principle (see for example \cite{GS}) that Erd\H{o}s--Kac type results can be deduced from level of distribution information.
 
Let 
\begin{equation}\label{eq:Ck}
C_k := \EE X^k,\qquad 
    C'_k := \EE |X|^k,
\end{equation}
 where $X\sim N(0,1)$. It is known that for \eqref{eq:prob} to hold it suffices to show that
\begin{equation}\label{eq:mom}
\EE\left( \frac{\omega(X_n)-2 \log \log n}{(2\log \log n)^{1/2}}\right)^k \to C_k, \qquad n \to \infty,
\end{equation}
holds for every integer $k \ge 1$ \cite[Ch.~3.3.5]{Durrett}. To approach \eqref{eq:mom} we use
\begin{proposition}\cite{GS}\label{prop:gs}
Let $f \colon\NN \to \RR_{\ge 0}$ be an arithmetic function and  $\mathcal{P}$ be a squarefree integer. Let $h \colon\NN \to \RR_{\ge 0}$ be a multiplicative function with $h(p) \le p$ for $p | \mathcal{P}$. For $a | \mathcal{P}  $ define $r_a$ via
\[ \sum_{m \le n: \, a \mid m}f(m) =\frac{h(a)}{a}\sum_{m \le n}f(m) + r_a.\]
Let
\[ \mu_{\mathcal{P}} = \sum_{p \mid \mathcal{P}} \frac{h(p)}{p}, \qquad \sigma^2_{\mathcal{P}} =  \sum_{p \mid \mathcal{P}} \frac{h(p)}{p}\left(1-\frac{h(p)}{p}\right).\]
Uniformly for $1 \le k \le \sigma_{\mathcal{P}}^{2/3}$ we have
\begin{equation}
\sum_{m=1}^{n} f(m) \big(\sum_{p \mid \mathcal{P}} \mathbf{1}_{p \mid m} -\mu_{\mathcal{P}}\big)^k =\sigma_{\mathcal{P}}^k \sum_{m=1}^{n} f(m) \big( C_k  + O( (k^{3/2}\sigma_{\mathcal{P}}^{-1})^{1+\mathbf{1}_{2 \mid k}}C'_k)\big)+ O\bigg(\mu_{\mathcal{P}}^k  \sum_{\substack{\omega(a) \le k\\ a \mid \mathcal{P} }} |r_a|\bigg),
\end{equation}
where $C_k$, $C'_k$ are given by \eqref{eq:Ck}. 
\end{proposition}
For $f \equiv 1$ Proposition~\ref{prop:gs} is \cite[Prop.~3]{GS} and the proof works the same way for general $f \ge 0$. We derive from Proposition~\ref{prop:gs} the following criterion, to be proved below.
\begin{cor}\label{cor:gs}
For every sufficiently large $n$ let $Y_{n}$ be a random variable taking values in $\{1,2,\ldots,n\}$ and set
\[ I_n:=(\exp(\exp(\log^{1/3}\log  n)), n^{\exp(-\log^{1/3}\log n)}), \quad T:=n^{1/\log^{1/3} \log n}, \quad  \mathcal{P} := \prod_{p \in I_n}p.\]
Fix $k \ge 1$. Suppose that
\begin{equation}
\label{eq:cond2}
\sum_{\substack{\omega(a) \le k\\ a\mid \prod_{p < T}p / \mathcal{P}}} \PP( a \textup{ divides } Y_n) =o\big( (\log \log n)^{k/2}\big)
\end{equation}
as $n \to \infty$, and that for some multiplicative function $g\colon \mathbb{N} \to [0,1]$,
\begin{equation}
\label{eq:cond1} \sum_{\substack{\omega(a) \le k\\a \mid \mathcal{P}}}\bigg| \PP(a \textup{ divides }Y_{n})-g(a)\bigg| =o\big( (\log \log n)^{-k/2}\big)
\end{equation}
as $n \to \infty$. If $\sum_{p \le x} g(p)= \theta \log \log x +O(1)$ for some $\theta>0$ and $\sum_{p \le x} g^2(p)=O(1)$ then
\begin{equation}\label{eq:kthmom}
\EE \bigg( \frac{\omega(Y_{n})-\theta \log \log n}{(\theta \log \log n)^{1/2}}\bigg)^k \to C_k
\end{equation}
as $n \to \infty$.
\end{cor}
\begin{proof}
We consider a given $n$ and apply Proposition~\ref{prop:gs} with $f(m) =\PP(Y_n=m)$, $h(m)=g(m)m$. Define
\begin{equation}
\omega_{\mathcal{P}}(n) = \sum_{p \mid \mathcal{P}} \mathbf{1}_{p \mid n},\quad \mu = \sum_{p \mid \mathcal{P}} g(p),\quad \sigma^2 = \sum_{p \mid \mathcal{P}} g(p)(1-g(p)), \quad R_a=\PP(a \textup{ divides }Y_{n})-g(a).
\end{equation}
We find that 
\begin{equation}\label{eq:app}
	\sum_{m=1}^{n} \PP(Y_n=m)\left(\frac{\omega_{\mathcal{P}}(m)-\mu}{\sigma}\right)^k =  \sum_{m=1}^{n} \PP(Y_n=m)\big( C_k + O( \sigma^{-1-\mathbf{1}_{2 \mid k}})\big)+  O\bigg(  \frac{\mu^k}{\sigma^k} \sum_{\substack{\omega(a) \le k\\ a \mid \mathcal{P}}} |R_a|\bigg)
\end{equation}
holds for every fixed $k \ge 1$ (implied constants may depend on $k$). We claim 
\begin{equation}\label{eq:omegamusigma}
\mu, \sigma^2 = \theta \log \log n+O(\log^{1/3}\log n).
\end{equation}
Indeed, by our assumptions on $\sum_{p \le x}g(p)$ and $\sum_{p \le x}g^2(p)$,
\[0\le \mu-\sigma^2 = \sum_{p \mid \mathcal{P}} g^2(p) \le \sum_{p \le n} g^2(p)=O(1)\]
and 
\begin{align}
\mu &= \sum_{p< n^{\exp(-\log^{1/3}\log n)}} g(p) - \sum_{p\le \exp(\exp(\log^{1/3}\log n))}g(p) \\
&= \theta \log \log (n^{\exp(-\log^{1/3}\log n)}) +O(1)-\theta \log\log (\exp(\exp(\log^{1/3}\log n))) + O(1) \\
&= \theta \log \log n + O(\log^{1/3}\log n).
\end{align}
We express \eqref{eq:app} as
\begin{align}\label{eq:diff} \EE\left(\frac{\omega_{\mathcal{P}}(Y_n)-\mu}{\sigma}\right)^k -C_k \ll  (\log \log n)^{-1/2} + (\log \log n)^{k/2}  \sum_{\substack{\omega(a) \le k\\ a \mid \mathcal{P}}} |R_a|.
\end{align}
By \eqref{eq:cond1}, the right-hand side of \eqref{eq:diff} goes to $0$ as $n \to \infty$. We have the identity
\[ \frac{\omega(Y_n)-\theta\log \log n}{(\theta \log \log n)^{1/2}} = A_1 A_2 + A_3 + A_4 + A_5\]
for
\begin{align}
A_1 &= \frac{\sigma}{(\theta \log \log n)^{1/2}}, \qquad  A_2 = \frac{\omega_{\mathcal{P}}(Y_n)-\mu}{\sigma},\qquad A_3=\frac{\mu-\theta\log \log n}{(\theta\log \log n)^{1/2}},\\
A_4 &= \frac{\sum_{p \in [2,T)\setminus I_n} \mathbf{1}_{p \mid Y_n}}{(\theta\log \log n)^{1/2}},\qquad A_5 = \frac{\sum_{p\in [T,n]}\mathbf{1}_{p \mid Y_n}}{(\theta\log \log n)^{1/2}}
\end{align}
for $T$ and $I_n$ as in the statement of the corollary. By \eqref{eq:omegamusigma}, the (constant) random variables $A_1$ and $A_3$ tend to $1$ and $0$, respectively. We have observed  $\EE A_2^k \to C_k$ by \eqref{eq:diff}. An integer $m$ can have at most $\log m/\log T$ prime factors of size at least $T$, implying $A_5 \ll (\log \log n)^{-1/6}$ and so $\EE A_5^k \to 0$ as $n \to \infty$. To handle $A_4$ we use \eqref{eq:cond2} which yields
\begin{equation}
\EE A_4^k = \frac{ \sum_{p_1,\ldots,p_k \in [2,T) \setminus I_n} \PP([p_1,\ldots,p_k]\textup{ divides } Y_n)}{(\theta\log \log n)^{k/2}} \ll \frac{\sum_{\substack{\omega(a) \le k\\ a \mid \prod_{p < T} p/ \mathcal{P}}} \PP(a \text{ divides } Y_n)}{(\log \log n)^{k/2}} =o(1).
\end{equation}
We have shown $\EE |A_3+A_4+A_5|^k \to 0$ and $\EE (A_1A_2)^k \to C_k$. By Hölder, \eqref{eq:kthmom} follows.
\end{proof}
Corollary~\ref{cor:gs} may also be proved by generalizing an argument of Billingsley \cite{Billingsley}.
To deduce Theorem~\ref{thm:e} from Corollary~\ref{cor:gs} we apply it with $Y_n=X_n$, $g(m)=2^{\omega(m)}/m$ and $\theta=2$. This reduces matters to showing that
\begin{equation}
\label{eq:cond2again}
\sum_{\substack{\omega(a) \le k\\ a\mid \prod_{p < T}p / \mathcal{P}}} \PP( a \textup{ divides } X_n) =o\big( (\log \log n)^{k/2}\big)
\end{equation}
and
\begin{equation}
\label{eq:cond1again} \sum_{\substack{\omega(a) \le k\\a \mid \mathcal{P}}}\bigg| \PP(a \textup{ divides }X_{n})-\frac{2^{\omega(a)}}{a}\bigg| =o\big( (\log \log n)^{-k/2}\big)
\end{equation}
hold as $n \to \infty$, for every $k \ge 1$, where $T$ and $\mathcal{P}$ are defined in Corollary~\ref{cor:gs}. We explain why Proposition~\ref{prop:titch} implies \eqref{eq:cond2again}--\eqref{eq:cond1again}. Applying Proposition~\ref{prop:titch} with  $S=\{ a : a\text{ divides }\prod_{p < T} p/ \mathcal{P}\}$, $B=1$ and $R = T^k$ we obtain
\begin{equation}\label{eq:primesagain} 
		\sum_{\substack{\omega(a) \le k \\ a \mid \prod_{p < T} p/ \mathcal{P}\\ (a,n)=1}}\bigg| \PP( a \textup{ divides } X_n) - \frac{2^{\omega(a)}}{a}\bigg| \ll_{k}\frac{\log R}{\log n}	\sum_{\substack{\omega(a)\le k\\ a \mid \prod_{p < T} p/ \mathcal{P} }}\frac{1}{a} +\frac{1}{\log n}
	\end{equation}
and the triangle inequality yields
\begin{align}\label{eq:primesagainagain} 
		\sum_{\substack{\omega(a) \le k \\ a \mid \prod_{p < T} p/ \mathcal{P}\\ (a,n)=1}} \PP( a \textup{ divides } X_n)  &\ll_{k} 	\sum_{\substack{\omega(a)\le k\\a \mid \prod_{p < T} p/ \mathcal{P}}}\frac{1}{a} +\frac{1}{\log n}\ll_k \big(1+ \sum_{p \mid \prod_{p < T} p /\mathcal{P}} \frac{1}{p}\big)^k \ll_k (\log \log n)^{k/3}
  \end{align}
using Mertens' theorem \cite[Thm.~2.7(d)]{MV}. Since 
\begin{equation}\label{eq:harmless}
\PP(a \textup{ divides }X_n)\ll n^{-1+o(1)}
\end{equation}
holds uniformly for $a$ with $(a,n)>1$,  \eqref{eq:cond2again} is obtained. To obtain \eqref{eq:cond1again} we apply Proposition~\ref{prop:titch} with $S =\{ a:a \text{ divides } \mathcal{P}\}$, $B=1$ and $R=(\max I_n)^k$ (for $I_n$ defined in Corollary~\ref{cor:gs}), obtaining
\begin{equation}\label{eq:primesagain2} 
		\sum_{\substack{\omega(a) \le k \\ a \mid \mathcal{P}\\ (a,n)=1}}\bigg| \PP( a \textup{ divides } X_n) - \frac{2^{\omega(a)}}{a}\bigg| \ll_{k} \frac{\log R}{\log n}	\sum_{\substack{\omega(a) \le k \\ a \mid \mathcal{P}}}\frac{1}{a} +\frac{1}{\log n}.
	\end{equation}
We conclude since $\log R /\log n\ll_k \exp(-\log^{1/3}\log n)$ holds by definition of $R$ and
\[ \sum_{\substack{\omega(a) \le k \\ a \mid \mathcal{P}}}\frac{1}{a} \le \big(1+\sum_{p\mid \mathcal{P}}\frac{1}{p}\big)^k\ll_k (\log \log n)^{k}\]
holds by Mertens' theorem \cite[Thm.~2.7(d)]{MV}. The condition $(a,n)=1$ in \eqref{eq:primesagain2} is harmless by \eqref{eq:harmless} and the fact that $\sum_{p \mid (\mathcal{P},n)}1/p$ goes to $0$ faster than any power of $\log \log n$ by the construction of $\mathcal{P}$.
\section{Proof of Proposition~\ref{prop:titch}}\label{sec:titch}
We have $\PP( a \textup{ divides } X_n) = A_a(n)/A_1(n)$ where \begin{equation}\label{eq:sumsa}
A_a(n):=\sum_{p<n:\, a \mid n-p}2^{\omega(n-p)}.
\end{equation}
Using \eqref{eq:elliott} we see that Proposition~\ref{prop:titch} will follow if we show
\begin{equation}\label{eq:primes2} 
	\sum_{\substack{a \le R,\, a \in S \\ \omega(a)\le k \\ (a,n)=1} }\mu^2(a)\bigg|A_a(n) - \phi(n)\frac{2^{\omega(a)}}{a}\bigg| \ll_{k,B}  \frac{\phi(n)\log R+n(\log \log n)^2}{\log n} \sum_{\substack{a \le R,\, a \in S \\ \omega(a) \le k\\ (a,n)=1}} \frac{\mu^2(a)}{a} + \frac{n}{\log^B n}.
\end{equation}
We write $2^{\omega(n-p)}$ as a divisor sum using $2^{\omega(m)}= \sum_{d\mid m}\mu^2(d)$ and $\mu^2(d) = \sum_{e^2 \mid d} \mu(e)$ and separate into different ranges:
\[2^{\omega(m)} = \sum_{m=ds}\mu^2(d)=\sum_{\substack{m=ds \\ d \le D}}\mu^2(d) +\sum_{\substack{m=ds \\ D<s<m/D}}\mu^2(d)+\sum_{\substack{m=e^2ds\\ s \le D \\ e\le V}} \mu(e)  + \sum_{\substack{m=e^2 ds \\ s \le D\\ e> V}} \mu(e),\]
with the choices
\begin{equation}\label{eq:choices}
V=\log n\qquad \text{and} \qquad D=\frac{\sqrt{n}}{R^2\exp((\log \log n)^2)},
\end{equation}
and $m$ is equal to $n-p$. We split $A_a(n)$ accordingly:
\begin{align}
A_a(n)&=\sum_{\substack{p<n\\ a\mid n-p}}\bigg(\sum_{\substack{n-p=ds\\ d \le D}}\mu^2(d)+\sum_{\substack{n-p=ds\\ D<s <(n-p)/D}}\mu^2(d)+\sum_{\substack{n-p=e^2ds\\ s\leq D\\ e\leq V}}\mu(e)+\sum_{\substack{n-p=e^2ds\\ s \le D\\ e>V}}\mu(e) \bigg)\\
&=:S_0(a)+S_1(a)+S_2(a)+S_3(a).
\end{align}
Writing $\pi(x;q,b)$ for the number of primes in $[1,x]$ that are congruent to $b$ modulo $q$,  we have
\begin{align}
 S_0(a)&=\sum_{d \le D} \mu^2(d) \pi(n-1;[d,a],n), \qquad S_1(a)=\sum_{D<s<(n-2)/D} \sum_{\substack{p<n-Ds\\p\equiv n([a,s])}}\mu^2((n-p)/s),\\
	\label{eq:divswitch2}   S_2(a)&=\sum_{\substack{ s\leq D\\ e\leq V}}\mu(e) \pi(n-1;[se^2,a],n), \qquad S_3(a)=\sum_{\substack{s \le D \\e>V}} \mu(e) \pi(n-1;[se^2,a],n).
\end{align}
Let
\[ \textup{Li}(m) = \int_{2}^{m}\frac{dt}{\log t}\]
be the logarithmic integral. Write $E(m;q,a)=\pi(m;q,a)-\tfrac{\textup{Li}(m) \mathbf{1}_{(q,a)=1}}{\varphi(a)}$ and accordingly $S_i(a)=M_i(a)+E_i(a)$ ($i=0,2$) for
\begin{align}
	M_0(a) &= \textup{Li}(n-1) \sum_{\substack{ d \le D\\(d,n)=1}}\frac{\mu^2(d)}{\phi([d,a])} , \qquad E_0(a) = \sum_{d \le D}\mu^2(d) E(n-1;[d,a],n),\\ 
\label{eq:M2def}	M_2(a) &= \textup{Li}(n-1)\sum_{\substack{s \le D \\ e \le V \\ (se,n)=1}} \frac{\mu(e)}{\phi([se^2,a])},\qquad
	E_2(a) = \sum_{\substack{s\le D\\ e \le V}} \mu(e)E(n-1;[se^2,a],n).
\end{align}
We first observe some basic properties of $\phi$. We have
\begin{equation}\label{eq:lcm}
	\phi(n_1) \phi(n_2) = \phi((n_1,n_2))\phi([n_1,n_2])
\end{equation}
and
\begin{equation}\label{eq:sub}
	\phi(n_1)\phi(n_2) \le \phi(n_1 n_2)
\end{equation}
for every $n_1,n_2 \ge 1$. We have
\begin{equation}\label{eq:log}
	\sum_{m \le A} 1/\phi(m) = C \log A + O(1)
\end{equation}
for $A \ge 1$ \cite[p.~42]{MV}. The estimate
\begin{equation}\label{eq:tail}
	\sum_{n > A} 1/\phi(n^2) \ll 1/A
\end{equation}
holds for each $A>0$, since it is equivalent to $\sum_{A \le n<2A}\tfrac{n^2}{\phi(n^2)}\ll A$ which follows from \cite[Thm.~2.14]{MV}. Moreover, $n\ge \phi(n)\gg n/\log \log n$ \cite[Thm.~2.9]{MV}. The following lemmas estimate the contributions of $S_i$ to the left-hand side of \eqref{eq:primes2} and will be proved below.
\begin{lem}\label{lem:bnds}
Suppose $1 \le R \le n^{1/100}$. Fix $k \ge 1$. Then for squarefree $1 \le a \le R$ with $\omega(a) \le k$ and $(a,n)=1$,
		\begin{align}
|S_1(a)| +|S_3(a)| \ll_k \frac{\phi(n)}{a\log n} (\log R + (\log \log n)^2).
	\end{align}
\end{lem}
\begin{lem}\label{lem:bv}
Suppose $1 \le R \le n^{1/100}$. Then for every $B>0$,
		 \[\sum_{a \le R} (|E_0(a)|+ |E_2(a)|)  \ll_B \frac{n}{\log^B n}.\]
\end{lem}
\begin{lem}\label{lem:mts}
For squarefree $1 \le a \le n$ with $\omega(a)\le k$,
\[ M_2(a)- M_0(a) \ll_k \frac{n}{a \log n}.\]
\end{lem}
\begin{lem}\label{lem:mts2}
Fix $k \ge 1$. Uniformly for $D\ge n^{1/3}$ and squarefree $1\le a\le n^{1/100}$  with $(a,n)=1$ and $\omega(a) \le k$,
	\begin{equation}\label{eq:final}
		\sum_{\substack{ d\le D\\ (d,n)=1}} \frac{\mu^2(d)}{\phi([d,a])} = \frac{2^{\omega(a)}}{a}\left( \frac{\phi(n)}{n} \log D  + O_k((\log \log n)^2 )\right)-\frac{\phi(n)}{n}\frac{\Lower_a}{\phi(a)}.
	\end{equation}
\end{lem}
Note that $M_0(a)$ coincides with $\textup{Li}(n-1)$ times the left-hand side of \eqref{eq:final}. Applying the last 4 lemmas and recalling our choice of $D$ in \eqref{eq:choices} we obtain
\begin{equation}\label{eq:primes222} 
	\sum_{\substack{a \le R,\, a \in S\\ \omega(a)\le k \\ (a,n)=1} }\mu^2(a)\bigg|A_a(n) - \phi(n)\frac{2^{\omega(a)}}{a}+\textup{Li}(n-1)\frac{2D_a\phi(n)}{n\phi(a)}\bigg| \ll_{k,B} \frac{\phi(n)\log R+n(\log \log n)^2 }{\log n}\sum_{\substack{ a \le R,\, a \in S \\ \omega(a) \le k\\(a,n)=1}} \frac{\mu^2(a)}{a} + \frac{n}{\log^B n}
\end{equation}
for $1 \le R \le n^{1/100}$. The contribution of the lower order term $\textup{Li}(n-1)2D_a \phi(n)/(n\phi(a) )$ can be absorbed in the right-hand side of \eqref{eq:primes222} using \eqref{eq:dabnd}, giving \eqref{eq:primes2}. This will complete the proof of Proposition~\ref{prop:titch}.

\subsection{Proof of Lemma~\ref{lem:bnds}}
Let $1 \le a \le R$ be a squarefree integer with $(a,n)=1$ and $\omega(a)\le k$. In estimating $S_1(a)$ we consider separately $(s,n)=1$ and $(s,n)>1$:
\begin{equation} 
|S_1(a)| \le \sum_{D<s<\frac{n}{D}}\pi(n-1;[a,s],n)   = \sum_{\substack{D<s<\frac{n}{D}\\(s,n)=1}}\pi(n-1;[a,s],n)+\sum_{\substack{D<s<\frac{n}{D}\\(s,n)>1}}\pi(n-1;[a,s],n) =:T_1 + T_2.
\end{equation}
To bound $T_2$, note $\pi(n-1;[a,s],n) \le \sum_{p \mid n} 1 \ll \log n$ holds for $(s,n)>1$, and so
\[ T_2 \ll \frac{n\log n}{D}\]
which is $\ll n/(a \log n) \ll \phi(n)(\log \log n)^2/(a\log n)$ by our choice of $D$, and the assumption $a \le R$. To bound $T_1$ we use the Brun--Titchmarsh Theorem:
\begin{equation}\label{eq:appBT} T_1 \ll\frac{n}{\log n}\sum_{\substack{ D<s<\frac{n}{D}\\(s,n)=1}} \frac{1}{\phi([a,s])}.
\end{equation}
The theorem saves $\log n$ since $[a,s] <R n/D \le n^{3/4}$ by our choice of $D$, and the assumption $a \le R$. To estimate the sum in \eqref{eq:appBT} we let $g=(a,s)$ and change the order of summation:
\begin{equation}
\sum_{\substack{D<s<\frac{n}{D}\\(s,n)=1}} \frac{1}{\phi([a,s])} \le  \frac{1}{\phi(a)}  \sum_{g \mid a}\sum_{\substack{D<s<\frac{n}{D}\\ g \mid s\\(s,n)=1}} \frac{\phi(g)}{\phi(s)} \le \frac{1}{\phi(a)}  \sum_{g \mid a}  \sum_{\substack{D/g<s'<\frac{n}{Dg}\\ (s',n)=1}} \frac{1}{\phi(s')} 
 \end{equation}
 where we used \eqref{eq:lcm} in the first inequality and \eqref{eq:sub} in the second one. By \cite[p.~43]{MV},  
\begin{equation}\label{eq:sophis}
\sum_{\substack{B< m \le A\\(m,n)=1}}1/\phi(m) \ll \frac{\phi(n)}{n} \log (A/B) + 2^{\omega(n)}(\log B)/B
\end{equation}
holds for $A \ge B \ge 2$. Applying \eqref{eq:sophis} with $(A,B)=(n/(Dg),D/g)$ we find that
\[ T_1 \ll \frac{n}{\log n} \frac{1}{\phi(a)}\frac{\phi(n)}{n}\log (2 + \frac{n}{D^2})\sum_{g \mid a} 1 \ll_k \frac{\phi(n)}{a\log n}\log (2 + \frac{n}{D^2}) \]
since $a$ is squarefree and $\omega(a)\le k$, which is acceptable by our choice of $D$. It remains to estimate $S_3(a)$. We consider separately $(se,n)=1$ and $(se,n)>1$:
\begin{equation} 
|S_3(a)| \le \sum_{\substack{s \le D\\ e> V\\ (se,n)=1}} \mu^2(e)\pi(n-1;[se^2,a],n)+\sum_{\substack{s \le D\\ e> V\\ (se,n)>1}} \mu^2(e)\pi(n-1;[se^2,a],n)=:T_3 + T_4.
\end{equation}
To bound $T_4$ we observe $\pi(n-1;[se^2,a],n) \le \sum_{p \mid n} 1 \ll \log n$ holds when $(se,n)>1$, which leads to
\[ T_4 \ll \log n \sum_{s \le D} \sum_{e \le \sqrt{n/s}} 1 \ll \log n \sqrt{nD} \]
which is $\ll n/(a \log n) \ll \phi(n)(\log \log n)^2/(a \log n)$ due to our choice of $D$, and since $a \le R$. To bound $T_3$ we use the Brun--Titchmarsh Theorem when $e \le  V \log n$ and otherwise use $\pi(n-1;m,n) \le n/m\le n/\phi(m)$:
\begin{equation}\label{eq:twosums} 
T_3 \ll \frac{n}{\log n}\sum_{\substack{s \le D \\ V\log n \ge e> V}} \frac{\mu^2(e)}{\phi([se^2,a])} + n\sum_{\substack{s \le D \\ e> V \log n}} \frac{\mu^2(e)}{\phi([se^2,a])}.
\end{equation}
The theorem saves $\log n$ when $e \le V\log n$ because $[se^2,a] \le D(V \log n)^2 R \le n^{1/2}$ by our choices of $V$ and $D$, and since $a \le R$. Let $g=(a,se^2)$ and observe there are coprime $g_1$ and $g_2$ such that $g=g_1g_2$, $g_1 \mid s$ and $g_2 \mid e$. It follows that
\begin{equation}\label{eq:sumlargev}
\sum_{\substack{s \le D \\ e> V'}} \frac{1}{\phi([se^2,a])} \le   \sum_{g_1 g_2 \mid a} \sum_{g_1 \mid s\le D} \frac{\phi(g_1)}{\phi(s)} \sum_{g_2 \mid e > V'} \frac{\phi(g_2)}{\phi(e^2)\phi(a)} \le \frac{\log n}{V'}   \sum_{ g_1 g_2 \mid a }\frac{1}{\phi(a)} \ll_k  \frac{\log n}{V'} \frac{1}{a}
\end{equation}
for every $V' \ge 1$. Here we used \eqref{eq:lcm}-\eqref{eq:tail}. Applying \eqref{eq:sumlargev} with $V'=V,V \log n$ we find $T_3 \ll n/(a\log n) \ll \phi(n)(\log \log n)^2/(a \log n)$ since $V=\log n$.
\subsection{Proof of Lemma~\ref{lem:bv}}
By Cauchy--Schwarz and the trivial bound $E(n-1;b,n) \ll n/b$, we have
\begin{multline}
\sum_{a \le R}|E_0(a)| \le \sum_{b \le DR} \tau(b)^2 |E(n-1;b,n)| \\
\le \sqrt{\sum_{b \le DR} \tau(b)^4 |E(n-1;b,n)|} \sqrt{\sum_{b \le DR}  |E(n-1;b,n)|} \ll \sqrt{n}  (\log n)^{O(1)} \sqrt{\sum_{b \le DR}  |E(n-1;b,n)|}
\end{multline}
where $b$ stands for $[d,a]$ and $\tau$ is the divisor function.  Since $DR \ll_{B} \sqrt{n}(\log n)^{-B}$ for every $B>0$, the Bombieri--Vinogradov Theorem \cite[Lemma]{ElliottH} implies that 
\[ \sum_{b \le DR}  |E(n-1;b,n)| \ll_B n(\log n)^{-B}\]
and so 
\[\sum_{a \le R}|E_0(a)|  \ll_B n(\log n)^{-B}\]
for every $B>0$. Similarly, letting $b:=[se^2,a]$,
\[\sum_{a \le R}|E_2(a)| \le \sum_{b \le DV^2R} \tau(b)^3 |E(n-1;b,n)| \ll_B n(\log n)^{-B}.\]
\subsection{Proof of Lemma~\ref{lem:mts}}
We complete the $e$-range in $M_2(a)$, obtaining
\begin{align}
	M_2(a) &=  \textup{Li}(n-1) \sum_{\substack{e \ge 1 \\ s \le D \\(se,n)=1}}\frac{\mu(e)}{\phi([se^2,a])} +  M_{2,1}(n)
\end{align}
where, by \eqref{eq:lcm}--\eqref{eq:sub},
\[ M_{2,1}(n) \ll \frac{n}{\log n} \sum_{ \substack{s \le D \\ e> V}} \frac{\mu^2(e)}{\phi([se^2,a])} \le \frac{n}{\phi(a)\log n} \sum_{ \substack{s \le D \\ e> V}} \frac{\mu^2(e)\phi((se^2,a))}{\phi(s)\phi(e^2)}.\]
Since $a$ is squarefree, if we let $g=(se^2,a)$, then there are coprime $g_1,g_2$ (not necessarily unique) such that $g=g_1g_2$, $g_1 \mid e$ and $g_2 \mid s$, and so (using \eqref{eq:log}--\eqref{eq:tail})
\[ M_{2,1}(n) \ll \frac{n}{\phi(a)\log n} \sum_{g_1 g_2 \mid a} \sum_{ g_1 \mid e> V} \frac{\mu^2(e)\phi(g_1)}{\phi(e^2)}\sum_{g_2 \mid s \le D} \frac{\phi(g_2)}{\phi(s)} \ll_k \frac{n}{aV} =\frac{n}{a \log n}\]
when $\omega(a)\le k$. Now write
\[  \textup{Li}(n-1) \sum_{\substack{e \ge 1 \\ s \le D \\(se,n)=1}}\frac{\mu(e)}{\phi([se^2,a])} = M_0(a) + M_{2,2}(a)\]
where
\[M_{2,2}(a)=\textup{Li}(n-1) \sum_{\substack{D<se^2\le De^2 \\(se,n)=1}}\frac{\mu(e)}{\phi([se^2,a])}\ll \frac{n}{\log n}  \sum_{D<se^2 \le De^2}\frac{\mu^2(e)}{\phi([se^2,a])}.\]
Using the same $g$, $g_1$ and $g_2$ as before,
\begin{align}
\label{eq:mtl}	\sum_{D < se^2 \le De^2} \frac{\mu^2(e)}{\phi([se^2,a])}\le  \frac{1}{\phi(a)} \sum_{e \ge 1} \frac{1}{\phi(e^2)} \sum_{D/e^2 < s \le D} \frac{\phi(g)}{\phi(s)}\le \frac{1}{\phi(a)} \sum_{g_1 g_2 \mid a} \sum_{g_1 \mid e} \frac{\mu^2(e)\phi(g_1)}{\phi(e^2)} \sum_{\substack{D/e^2 < s \le D\\ g_2 \mid s }} \frac{\phi(g_2)}{\phi(s)} \ll_k \frac{1}{a}
\end{align}
where we used \eqref{eq:lcm}--\eqref{eq:tail}.
\subsection{Proof of Lemma~\ref{lem:mts2}}
For $a$ coprime to $n$ let
\[ f_{n,a}(d) := \frac{\mu^2(d)\phi(a)}{\phi([d,a])}\mathbf{1}_{(d,n)=1}\]
which is multiplicative in $d$. Elliott and Halberstam established \cite[p.~202]{ElliottH}
\begin{equation}\label{eq:EH} \sum_{d\le D} f_{n,1}(d)= \frac{\phi(n)}{n} \log D  + O((\log \log n)^2)
\end{equation}
for $n \le D^5$. We have, for $a$ coprime to $n$, $f_{n,a} = f_{n,1} * g_{a}$ where $g_{a}(d)=g(d)\mathbf{1}_{ d \mid a^{\infty}}$ and $g$ is defined in \eqref{eq:gdef}. Hence, by \eqref{eq:EH},
\begin{align} \label{eq:mt}
	\sum_{\substack{ d\le D\\ (d,n)=1}} \frac{\mu^2(d)}{\phi([d,a])} = \sum_{d\le D} \frac{f_{n,a}(d)}{\phi(a)}= \frac{\phi(n)}{n\phi(a)}\sum_{d \le D/n^{1/5}} g_{a}(d) \log (D/d)+  O\left(\frac{E}{\phi(a)}\right)
\end{align}
for 
\begin{equation}\label{eq:Edef} E := \sum_{d \le D} |g_{a}(d)| (\log \log n)^2+\sum_{d> D/n^{1/5}}|g_{a}(d)|\log d.
\end{equation}
To bound the first sum in $E$ we use $\sum_{d\ge 1} |g_{a}(d)|=2^{\omega(a)-\mathbf{1}_{2 \mid a}}\ll_k 1$. We turn to the second sum in $E$. If $d$ is in the support of $g_{a}$ then $d$ is odd and
\[ |g_{a}(d)| \le \prod_{p \mid d} p \prod_{p^i \mid \mid d} p^{-i} \prod_{p^i \mid \mid d}\left(1+\frac{1}{p-1}\right)^{i-1} \le \frac{a}{d}\prod_{p^i \mid \mid d}\left(1+\frac{1}{p-1}\right)^{i-1}\le \frac{a}{d}  \left(\frac{3}{2}\right)^{\log d/\log 3}.\]
If $\omega(a)\le k$ then there are $\ll_k (\log T)^k$ integers $d \in [T,2T)$ with $d \mid a^{\infty}$, so
\begin{equation}\label{eq:dyadic}
	\sum_{d\in [T,2T)} |g_{a}(d)|  \ll_k \frac{a(\log T)^k}{T} T^{\log(3/2)/\log 3} .
\end{equation}
This implies that under our assumptions, $E/\phi(a) \ll_k (\log \log n)^2/a$.
With the same $E$ as in \eqref{eq:Edef}, we express the sum in the right-hand side of \eqref{eq:mt}  as
\[ \sum_{ d \le D/n^{1/5}} g_{a}(d) \log (D/d) = \log D\prod_{p \mid a} \big( \sum_{i=0}^{\infty} g_{a}(p^i)\big) - \sum_{d\ge 1} g_{a}(d) \log d + O(E).\]
The $p$-product is $2^{\omega(a)}\phi(a)/a$ and the $d$-sum is $\Lower_a$.

\section{Proof of Proposition~\ref{prop:titchstronger}}\label{sec:titchstronger}
As in the proof of Proposition~\ref{prop:titch}, we work with $A_a(n)$ (defined in \eqref{eq:sumsa}) and reduce the lemma to proving
\begin{equation}\label{eq:primes22} 
	\sum_{\substack{1 \le a \le n^{\eta}\\ \omega(a)\le k \\ (a,n)=1} }\mu^2(a)\bigg|A_a(n) - \phi(n)\bigg(\frac{2^{\omega(a)}}{a}-\frac{2\Lower_a}{\phi(a)\log n}\bigg)\bigg| \ll_k \frac{n}{\log n}(\log \log n)^{O_k(1)}.
\end{equation}
We write $2^{\omega(n-p)}$ as a divisor sum, separate the contribution of large square divisors, and apply divisor switching. We have, for $m=n-p$,
\[2^{\omega(m)} = \sum_{m=ds}\mu^2(d) =\sum_{m=e^2ds}\mu(e)=\sum_{\substack{m=e^2ds\\ e>V}}\mu(e)+\sum_{\substack{m=e^2ds\\ e\leq V\\ d\leq D}}\mu(e)+\sum_{\substack{m=e^2ds\\ e\leq V\\ s<  m/(e^2 D)}}\mu(e)\]
with the choices \[V=n^{2\eta} \qquad \text{and} \qquad D=n^{1/2+2\eta}.\] Then, given $a\geq 1$, we split $A_a(n)$ accordingly:
\[    A_a(n)=\sum_{\substack{p<n\\ a|n-p}}\bigg(\sum_{\substack{n-p=e^2ds\\ e>V}}\mu(e)+\sum_{\substack{n-p=e^2ds\\ e\leq V\\ d\leq D}}\mu(e)+\sum_{\substack{n-p=e^2ds\\ e\leq V\\ s<  (n-p)/(e^2 D)}}\mu(e) \bigg)=:S_0(a)+S_1(a)+S_2(a),\]
so that 
\begin{align}
 S_0(a)&=\sum_{\substack{V<e\leq \sqrt{n}\\ d\leq n}}\mu(e)\pi(n-1;[de^2,a],n),\qquad S_1(a)=\sum_{\substack{ e\leq V\\ d\leq D}}\mu(e) \pi(n-1;[de^2,a],n), \\
 S_2(a)&=\sum_{\substack{e\leq V\\ s<  (n-2)/(e^2 D)}}\mu(e)\pi(n-se^2D-1;[se^2,a],n).
\end{align}
Write $E(m;q,a)=\pi(m;q,a)-\frac{\textup{Li}(m) \mathbf{1}_{(q,a)=1}}{\varphi(a)}$ and accordingly
\begin{align*}
	S_i(a)&=N_i(a)+E_i(a) \qquad (i=1,2)
\end{align*}
for
\begin{align}
	N_1(a) &= \textup{Li}(n-1) \sum_{\substack{e \le V \\ d \le D\\(de,n)=1}}\frac{\mu(e)}{\phi([de^2,a])} , \qquad E_1(a) = \sum_{\substack{e \le V \\ d \le D\\(de,n)=1}}\mu(e) E(n-1;[de^2,a],n),\\ 
	N_2(a) &= \sum_{\substack{e \le V \\ s<(n-2)/(e^2 D)\\(se,n)=1}} \frac{\mu(e)\textup{Li}(n-se^2D-1)}{\phi([se^2,a])},\qquad
	E_2(a) = \sum_{\substack{e \le V \\ s<(n-2)/(e^2 D)\\(se,n)=1}} \mu(e)E(n-se^2D-1;[se^2,a],n).
\end{align}
The rest of the proof is separated into four lemmas, proved below, which together establish \eqref{eq:primes22}. This will complete the proof of Proposition~\ref{prop:titchstronger}.
\begin{lem}\label{lem:S0}
	We have
	$\sum_{a \le n^{\eta}}|S_0(a)|\ll n^{1-\eta}$.
\end{lem}
\begin{lem}\label{lem:E1}
	For every $B>0$, $\sum_{a \le n^{\eta},\,(a,n)=1}|E_1(a)|\ll_B  n(\log n)^{-B}$.
\end{lem}
\begin{lem}\label{lem:E2}
	For every $B>0$,
	$\sum_{a \le n^{\eta},\,(a,n)=1}|E_2(a)|\ll_B  n(\log n)^{-B}$.
\end{lem}
\begin{lem}\label{lem:m}
	Fix $k \ge 1$. Suppose $1 \le a \le n^{\eta}$ is squarefree with $\omega(a) \le k$ and $(a,n)=1$. We have
	\begin{align}\label{eq:M1M2} N_{1}(a)+N_{2}(a) =\phi(n) \bigg( \frac{2^{\omega(a)}}{a} - \frac{2\Lower_a}{\phi(a)\log n}\bigg) +O_k\bigg( \frac{n (\log \log n)^2 }{a \log n} \bigg).
	\end{align}
\end{lem}
\subsection{Proof of Lemma~\ref{lem:S0}}
The bound $\pi(n-1;q,n)\leq n/q$ gives
\[    \left|S_0(a)\right|\leq  \sum_{\substack{e>V\\ d\leq n}} \frac{n}{[de^2,a]} \le \sum_{\substack{e>V\\ d\leq n}} \frac{n}{de^2} \ll \frac{n \log n}{V}\]
which implies the result.
\subsection{Proof of Lemma~\ref{lem:E2}}
Letting $d:=[se^2,a]$, we have
\begin{align*}
	\sum_{\substack{a \le n^{\eta}\\(a,n)=1}}|E_2(a)|=\sum_{\substack{a \le n^{\eta}\\(a,n)=1}}\bigg|\sum_{\substack{e\leq V\\ s< (n-2)/(e^2D)\\(se,n)=1} }\mu(e) E(n-se^2D-1;[se^2,a],n)\bigg|\le \sum_{\substack{d\leq n^{1/2 - \eta} \\ (d,n)=1}}\tau^3(d)\max_{y\leq n} \left|E(y;d,n) \right|.
\end{align*}
The conclusion follows by applying Cauchy--Schwarz and the Bombieri--Vinogradov Theorem \cite[Lemma]{ElliottH} (cf.~the proof of Lemma~\ref{lem:bv}).
\subsection{Proof of Lemma~\ref{lem:E1}}
As $E_1(a)$ requires us to break the $1/2$-barrier in the level of distribution, it is the more delicate error term. In the previous two lemmas, $\eta$ can be close to $1$, but not here. We have
\begin{align*}
	\sum_{\substack{a \le n^{\eta}\\ (a,n)=1}}|E_1(a)|&=\sum_{\substack{a \le n^{\eta}\\ (a,n)=1}}\bigg|\sum_{\substack{e \le V\\ (e,n)=1}} \mu(e)\sum_{\substack{d \le D\\ (d,n)=1}}E(n-1;[de^2,a],n)\bigg|\\ 
	&= \sum_{\substack{a \le n^{\eta}\\ (a,n)=1}}\bigg|\sum_{\substack{e \le V\\ (e,n)=1 \\ b\mid (ae)^\infty}} \mu(e)\sum_{\substack{d \le D\\  b \mid d \\ (d/b,nae)=1}}E(n-1;[de^2,a],n)\bigg|\\
	&=\sum_{\substack{a \le n^{\eta}\\ (a,n)=1}}\bigg|\sum_{\substack{e \le V\\ (e,n)=1 \\ b\mid (ae)^\infty}} \mu(e)\sum_{\substack{d' \le D/b \\ (d',nae)=1}}E(n-1;d'[be^2,a],n)\bigg|,
\end{align*}
where we substituted $d=bd'$. Let $E_{1,1}$ denote the contribution to the above with $b\le 
J=n^{2\eta}$ and $E_{1,2}$ the remainder. 
For $E_{1,1}$ we apply the triangle inequality to arrive at
\begin{align*}
	E_{1,1}\leq \sum_{\substack{a \le n^{\eta}\\ (a,n)=1}} \sum_{\substack{e \le V\\ (e,n)=1}}\mu^2(e)\sum_{\substack{ b\leq J\\ b \mid (ae)^{\infty}} }\bigg|\sum_{\substack{d' \le D/b\\ (d',nae)=1}}E(n-1;d'[be^2,a],n)\bigg|.
\end{align*}
To make the summation range of $d'$ independent of $b$, we apply a standard finer-than-dyadic decomposition on the $d'$ summation  (see e.g.~the proof of \cite[Thm.~1.1]{Assing}). More precisely, set $\lambda=1+(\log n)^{-B-100}$ and decompose $E_{1,1}$ into the contribution of $d' \leq D/J$ and at most $(\log n)^{B+101}$ sums with the additional condition $(D/J)\lambda^k < d'\leq (D/J) \lambda^{k+1}$ for some $k\geq 0$. Thus, by the triangle inequality,
\[\sum_{\substack{ b\leq J\\ b \mid (ae)^{\infty}} }\bigg|\sum_{\substack{d' \le D/b\\ (d',nae)=1}}\bigg| \leq \sum_{\substack{b\leq J\\ b \mid (ae)^{\infty} }}\bigg|\sum_{\substack{d'\leq D/J\\ (d',nae)=1}}\bigg|+ \sum_{\substack{b\leq J \\b \mid (ae)^{\infty}\\k\ge 0}} \bigg| \sum_{\substack{(D/J)\lambda^k < d'\leq (D/J) \lambda^{k+1} \\ d'\leq D/b\\(d',nae)=1}} \bigg|.\]
For each $b\le J$ there exists a unique $k_0(b)$ such that $(D/J)\lambda^{k_0(b)} < D/b \leq (D/J) \lambda^{k_0(b)+1} $. For $k<k_0(b)$ the condition $d'\leq D/b$ is implied by the other range condition ($d' \le (D/J)\lambda^{k+1}$) and for $k>k_0(b)$ the sum over $d'$ is empty. 
The term with $k=k_0(b)$ can be estimated trivially, indeed it is
\begin{align*}
    &\sum_{\substack{a \le n^{\eta}\\ (a,n)=1}} \sum_{\substack{e \le V\\ (e,n)=1}}\mu^2(e)\sum_{\substack{ b\leq J\\ b \mid (ae)^{\infty}\\ k=k_0(b)} }\bigg|\sum_{\substack{(D/J)\lambda^k < d'\leq (D/J) \lambda^{k+1} \\ d'\leq D/b\\ (d',nae)=1}}E(n-1;d'[be^2,a],n)\bigg|\\
    &\ll \sum_{a \le n^{\eta}} \sum_{e \le V}\sum_{\substack{ b\leq J\\ b \mid (ae)^{\infty}\\ k=k_0(b)} }  \sum_{\substack{(D/J)\lambda^k < d'\leq (D/J) \lambda^{k+1} }}\frac{n}{\varphi(d')\varphi([be^2,a])}\\
    &\ll  n(\log \log n)^2 \sum_{\substack{ a\leq n^\eta\\ b'\leq JV^2}} \frac{\tau(b')}{[b',a]} \sum_{\substack{(D/J)\lambda^{k_0(b)} < d'\leq (D/J) \lambda^{k_0(b)+1} }} \frac{1}{d'}\\
    &\ll  n (\log n)^{4} \log \lambda \ll (\log n)^{-B-50},
\end{align*}
recalling $\lambda=1+(\log n)^{-B-100}$. Here $b'$ stands for $be^2$, in the second inequality we used $\phi(m) \ge m/\log \log m$ \cite[Thm.~2.9]{MV} and in third one we used $\sum_{a\le T,\,b' \le S}(b',a)  \tau(b')\ll TS (\log (TS+2))^3$.

Consequently after some simple reshuffling of the terms with $k<k_0(b)$ with the help of the triangle inequality, we have reduced the required $E_{1,1}$ estimate for Lemma~\ref{lem:E1} to showing 
\begin{align}\label{eq:E11D'estimate}
    E_{1,1,D'}\ll n(\log n)^{-2B-101},
\end{align}
for any $D'\leq D$, where
\begin{align*}
	E_{1,1,D'}=\sum_{\substack{a \le n^{\eta}\\ (a,n)=1}} \sum_{\substack{e \le V\\ (e,n)=1}}\mu^2(e)\sum_{\substack{ b\leq J\\ b \mid (ae)^{\infty}} }\bigg|\sum_{\substack{d\leq D' \\ (d',nae)=1}}E(n-1;d'[be^2,a],n)\bigg|.
\end{align*}
If $[be^2,a]=g$ then $a$, $b$ and $e$ divide $g$ and the condition $(d',ae)=1$ is equivalent to $(d',g)=1$. Thus 
\[ E_{1,1,D'}\le  \sum_{\substack{g \le n^{8\eta}\\(g,n)=1}} \tau(g)^3\bigg| \sum_{\substack{d' \le D'\\(d',ng)=1}} E(n-1;d'g,n)\bigg|.\]
We quote \cite[Prop.~6.1]{Blomer} specialised to $a_2=c=c_0=1$.

\begin{proposition}\label{PrimeAp}
	There is an absolute constant $\varpi>0$ with the following property. 
	Let $x\geq 2$, $Q\ge 2$, $C\in \mathbb N$,  $a\in \mathbb{Z}\setminus \{0\}$  and a sequence $(\lambda_d)_d$ such that 
	\begin{align}
		|a|\leq x^{1+\varpi}, \quad |\lambda_d|\ll \tau(d)^C, \quad Q\leq x^{1/2+\varpi}
	\end{align} 
	hold. Then for any $B>0$ we have 
	\begin{align}
		\sum_{d\leq x^{\varpi}}\lambda_d\sup_{\substack{w \bmod d\\(w,d)=1}}\Big|\sum_{\substack{q\leq Q\\ (q,ad)=1}}\Big(\sum_{\substack{n\leq x\\n\equiv a\bmod q\\ n\equiv w\bmod d}}\Lambda(n)-\frac{1}{\phi(qd)}\sum_{\substack{n\leq x\\ (n, qd)=1}}\Lambda(n)\Big)\Big|\ll_{C,B} x(\log x)^{-B}.
	\end{align}
\end{proposition}
As commented after \cite[Prop.~6.1]{Blomer}, this holds for the characteristic function of the primes replacing the von Mangoldt function and so implies \eqref{eq:E11D'estimate}. To prove Lemma~\ref{lem:E1}, it remains to estimate $E_{1,2}$ and we claim
\begin{align}\label{eq:E12bound}
	E_{1,2}\ll n^{1-\eta/2}.
\end{align}
By the  trivial bound $E(n-1;q,a)\ll 1+n/q$ we have
\[  E_{1,2}\ll n\sum_{\substack{a\leq n^\eta \\ e\leq V \\ d'< D/J}} \frac{1}{d'} \sum_{\substack{J<b\leq D/d'\\ b\mid (ae)^\infty }}\frac{1}{[be^2,a]}.\]
Since $[be^2,a] = be^2a/(be^2,a)$, by dyadic decomposition to get \eqref{eq:E12bound} it suffices to show
\begin{equation}\label{eq:dyadicaeb}
	\sum_{\substack{a \sim A\\ e \sim V}} \sum_{\substack{b \sim J \\ b \mid (ae)^{\infty}}} (be^2,a) \ll(2+\log AV)^{C_{\varepsilon}} J^{\varepsilon} A^2 V
\end{equation}
for all $\varepsilon>0$. Here $C_{\varepsilon}>0$ depends only on $\varepsilon>0$ (later occurrences of $C_{\varepsilon}$ might have a different value), $a \sim A$ means $a \in [A,2A)$ and $e \sim V$, $b \sim J$ should be interpreted in the same way. Estimate~\eqref{eq:dyadicaeb} implies $E_{1,2} \ll n^{1+\eta}(\log n)^{C_{\varepsilon}}/J^{1-\varepsilon} \ll n^{1-\frac{\eta}{2}}$ and so ~\eqref{eq:E12bound} if $\varepsilon$ is sufficiently small.

To demonstrate \eqref{eq:dyadicaeb}, let $h=(be^2,a)$. If $h \mid be^2$ then $h=h_1h_2$ for some (not necessarily unique) $h_1 \mid b$ and $h_2 \mid e^2$. This implies $h'_2 \mid e$ for $h'_2 = \prod_{p^k \mid \mid h_2} p^{\lceil k/2\rceil}$. We can bound the left-hand side of \eqref{eq:dyadicaeb} by
\[ \le \sum_{h_1 h_2 \mid a \sim A} h_1 h_2  \sum_{\substack{h'_2 \mid e \sim V \\ b \sim J \\ h_1 \mid b \mid (ae)^{\infty}}} 1 \le\sum_{h_1 h_2  \mid a \sim A } h_1 h_2  \sum_{\substack{e' \sim V/h'_2 \\ b' \sim J/h_1 \\ b'  \mid (ae')^{\infty}}} 1 .\]
By Rankin's trick,
\[ \sum_{b \le B, \, b \mid m^{\infty}} 1 \le B^{\varepsilon}\prod_{p \mid m}(1-p^{-\varepsilon})^{-1} \le B^{\varepsilon} \exp(C_{\varepsilon} \sum_{p \mid m} p^{-\varepsilon}) \le B^{\varepsilon} \exp(C_{\varepsilon}\omega(m))\]
for every $\varepsilon>0$. Hence
\[\sum_{\substack{e' \sim V/h'_2 \\ b' \sim J/h_1 \\ b'  \mid (ae')^{\infty}}} 1 \ll \sum_{e' \sim V/h'_2} (J/h_1)^{\varepsilon} \exp(C_{\varepsilon} (\omega(a)+\omega(e'))) \ll (J/h_1)^{\varepsilon} V(2+\log V)^{C_{\varepsilon}}\exp(C_{\varepsilon} \omega(a))\]
and so, since $h_1 h_2 \ll A$,
\begin{align} \sum_{\substack{a \sim A\\ e \sim V}} \sum_{\substack{ b \sim J \\ b \mid (ae)^{\infty}}} (be^2,a) &\ll (2+\log V)^{C_{\varepsilon}} J^{\varepsilon}V \sum_{h_1 h_2 \mid a \sim A } h_1^{1-\varepsilon} h_2 \exp(C_{\varepsilon} \omega(a))\\
	&\ll (2+\log V)^{C_{\varepsilon}} J^{\varepsilon} AV\sum_{a \sim A} \exp(C_{\varepsilon}\omega(a)) \tau(a)^2 \ll( 2+\log AV)^{C_{\varepsilon}} J^{\varepsilon} A^2V .
\end{align}
This proves \eqref{eq:dyadicaeb}, and completes the proof of Lemma~\ref{lem:E1}.
\subsection{Proof of Lemma~\ref{lem:m}}
Let $1 \le a \le n^{\eta}$ be squarefree with $(a,n)=1$. The expression $N_1(a)$ matches $M_2(a)$ defined in \eqref{eq:M2def}, so
\[ N_1(a) = \textup{Li}(n-1) \sum_{\substack{ d \le D\\(d,n)=1}}\frac{\mu^2(d)}{\phi([d,a])} + O\left(\frac{n}{a \log n} \right)\]
by Lemma~\ref{lem:mts}, its proof and our choice of $V$. The $d$-sum is estimated in Lemma~\ref{lem:mts2}, giving
\[ N_1(a) = \phi(n) \left( \frac{2^{\omega(a)}}{a}\frac{\log D}{\log n} - \frac{D_a}{\phi(a)\log n}\right) + O_k \left(  \frac{n(\log \log n)^2}{a\log n} \right)\]
where we used \eqref{eq:dabnd}. It remains to deal with $N_2(a)$. By completing the $e$-range we have
\begin{align}
	N_2(a) =  \sum_{\substack{e \ge 1 \\ s<(n-2)/(e^2 D)\\(se,n)=1}} \frac{\mu(e)\textup{Li}(n-se^2D-1)}{\phi([se^2,a])} + O\left( \frac{n}{V}\right).
\end{align}
We now use  $\textup{Li}(n-se^2D-1) =  \textup{Li}(n-1) + O(se^2 D/\log n)$ to decompose $N_2(a)$ as
\[	N_2(a)= N_{2,1}(a)+N_{2,2}(a) +  O\left( \frac{n}{V}\right)\]
where 
\begin{align}
	N_{2,1}(a) = \textup{Li}(n-1) \sum_{\substack{se^2<(n-2)/D \\ (se,n)=1}} \frac{\mu(e)}{\phi([se^2,a])}=\textup{Li}(n-1) \sum_{\substack{m <(n-2)/D \\ (m,n)=1}} \frac{\mu^2(m)}{\phi([m,a])}
\end{align} 
and
\begin{align}
\label{eq:m222}	N_{2,2}(a) \ll \frac{D}{\log n} \sum_{\substack{se^2<(n-2)/D \\ (se,n)=1}} \frac{\mu^2(e)se^2}{\phi([se^2,a])}  = \frac{D}{\log n} \sum_{\substack{m< (n-2)/D\\(m,n)=1}} \frac{m 2^{\#\{p^2 \mid m\}}}{\phi([m,a])}.
\end{align} 
We let $g=(a,m)$ in the right-hand side of \eqref{eq:m222} and use \eqref{eq:lcm}--\eqref{eq:sub} to find
\[N_{2,2}(a) \ll \frac{D}{ \phi(a)\log n} \sum_{g \mid a}g\sum_{g \mid m < n/D }\frac{(m/g) 2^{\#\{p^2 \mid m\}}}{\phi(m/g)}.\]
Let $S(X):=\sum_{m \le X} (m/\phi(m))2^{\#\{p^2 \mid m\}}$. Note $S(X) \ll X$ by \cite[Thm.~2.14]{MV}. Hence
\[ N_{2,2}(a) \ll \frac{D}{\phi(a)\log n } \sum_{g \mid a} g2^{\omega(g)}S\left( \frac{n}{Dg}\right) \ll \frac{n}{\phi(a)\log n} 3^{\omega(a)} \ll_k \frac{n}{a \log n}.\]
To conclude we apply Lemma~\ref{lem:mts2} with $n/D$ in place of $D$ to estimate  $N_{2,1}(a)$, obtaining
\[ N_2(a) = \phi(n) \left( \frac{2^{\omega(a)}}{a}\frac{\log (n/D)}{\log n} - \frac{D_a}{\phi(a)\log n}\right) + O_k \left(  \frac{n(\log \log n)^2}{a\log n} \right).\]

\section{Elementary proof of (1.8)}\label{sec:motivation}
Let $X \sim N(0,1)$. By L\'{e}vy's continuity theorem \cite[Thm.~3.3.17]{Durrett}, \eqref{eq:uprime} is equivalent to
\[ \EE e^{it \left(\frac{\omega(\widetilde{X}_n)-2\log \log n}{(2 \log \log n)^{1/2} } \right)} = \big( \sum_{m \le n}2^{\omega(m)}\big)^{-1} \sum_{m \le n} 2^{\omega(m)}e^{it \left(\frac{\omega(m)-2\log \log n}{(2 \log \log n)^{1/2} } \right)} \to \EE e^{it X} = e^{-\frac{t^2}{2}}, \qquad t \in \RR,\]
a relation that is readily verified using the Selberg--Delange method, as observed by Elliott \cite[p.~1364]{Elliott}. We give an elementary proof of \eqref{eq:uprime} using Corollary~\ref{cor:gs}. One only needs to verify \eqref{eq:cond2} and \eqref{eq:cond1}, both of which follow from the next lemma, whose proof is a streamlined version of the arguments in \cite[\S3]{EG}.
\begin{lem}\label{lem:sa}
For positive integers $1 \le a \le n$ such that $\omega(a) \le k$ we have
	\[\PP(a \textup{ divides } \widetilde{X}_n) = \frac{2^{\omega(a)}}{a}\bigg(1 -\frac{\log a}{\log n}+ O_{k}\bigg(\frac{1}{\log n} + \sum_{p \mid a} \frac{1}{p}\bigg)\bigg).\]
\end{lem}
\begin{proof}
	Let $S_a(n) := \sum_{a \mid r \le n} 2^{\omega(r)}$. We have
	\begin{align}
		S_a(n) &= \sum_{m\le n/a} 2^{\omega(am)} = \sum_{d \mid a^{\infty}} 2^{\omega(ad)}\sum_{m' \le n/(ad), \, (m',a)=1} 2^{\omega(m')}=A+B
	\end{align}
	where $d$ stands for $(m,a^{\infty})$, $A$ is the contribution of $d=1$ and $B$ is the contribution of $d>1$. We have $0 \le B \le \sum_{1<d \mid a^{\infty}} 2^{\omega(ad)} S_1(n/(ad))$ by dropping the coprimality condition. For $A$ we have
	\[ A = 2^{\omega(a)}(\sum_{m'\le n/a} 2^{\omega(m')} -\sum_{1<d \mid a^{\infty}} 2^{\omega(d)} \sum_{m'' \le n/(ad), \, (m'',a)=1} 2^{\omega(m'')})\]
	and so, by dropping the coprimality condition $(m'',a)=1$,
	\[ 0 \le  2^{\omega(a)}\sum_{m \le n/a}2^{\omega(m)}- A \le \sum_{1<d \mid a^{\infty}}2^{\omega(a)+\omega(d)}\sum_{m \le n/(ad)} 2^{\omega(m)}.\]
	We have proved
	\begin{equation}\label{eq:sas1} \left|S_a(n)-2^{\omega(a)}S_1(n/a)\right|  \le \sum_{1<d \mid a^{\infty}}(2^{\omega(ad)}+2^{\omega(a)+\omega(d)})S_1(n/(da)).
	\end{equation}
	We divide \eqref{eq:sas1} by $S_1(n)$ and use $S_1(n) = \zeta(2)^{-1} n \log n ( 1 +O(1/\log n))$ \cite[p.~42]{MV}.
\end{proof}
\begin{rem}
Recall the notation \eqref{eq:defV}. Given $k \ge 1$ and a continuous function $f \colon \RR^k \to \CC$ such that the (closure of the) support of $f$ is $\subseteq\{ (x_1,\ldots,x_k): x_1,\ldots,x_k >0, \, x_1+\ldots+x_k < 1\}$, with further work one can compute 
\[ \lim_{n \to \infty} \EE \sum_{j_1,\ldots,j_k \text{ distinct}} f( V(\widetilde{X}_n)_{j_1},\ldots,V(\widetilde{X}_n)_{j_k})\]
using Lemma~\ref{lem:sa} and show that it coincides with 
\[ \EE \sum_{j_1,\ldots,j_k \text{ distinct}} f( L_{j_1},\ldots,L_{j_k})\]
where $(L_1,L_2,\ldots)$ is distributed according to $\textup{PD}(2)$. This in turn will imply that $V(\widetilde{X}_n)$ tends in distribution to $\textup{PD}(2)$, see \cite{arratia2014extensions} and especially Remark~12 of \cite{bharadwaj2024large}.
\end{rem}
\begin{acknowledgement}
We thank Brad Rodgers for discussions around Corollary~\ref{cor:PD} and useful suggestions, and the anonymous referees for corrections, suggestions as well useful comments which led us to a simplified proof of Elliott's conjecture using Proposition~\ref{prop:titch}.
\end{acknowledgement}

\bibliographystyle{abbrv}

\end{document}